\documentclass[12pt]{article}

\usepackage{latexsym}
\usepackage{amssymb}
\usepackage{graphicx}
\usepackage{color}
\usepackage{amsmath}
\usepackage{float}
\usepackage{mathrsfs} 

\usepackage{accents}
\newlength{\dhatheight}
\newcommand{\doublecheck}[1]{%
    \settoheight{\dhatheight}{\ensuremath{\check{#1}}}%
    \addtolength{\dhatheight}{-0.15ex}%
    \check{\vphantom{\rule{1pt}{\dhatheight}}%
    \smash{\check{#1}}}}

\newcommand{\doublehat}[1]{%
    \settoheight{\dhatheight}{\ensuremath{\hat{#1}}}%
    \addtolength{\dhatheight}{-0.25ex}%
    \hat{\vphantom{\rule{1pt}{\dhatheight}}%
    \smash{\hat{#1}}}}

\newtheorem{theorem}{Theorem}

\newtheorem{corollary}[theorem]{Corollary}

\newtheorem{lemma}[theorem]{Lemma}

\newtheorem{remark}[theorem]{Remark}

\newenvironment{proof}[1][Proof]{\textbf{#1.} }{\ \rule{0.5em}{0.5em}}
\floatstyle{ruled} \restylefloat{figure} \restylefloat{table}

\newcommand{\kom}[1]{}
\renewcommand{\kom}[1]{{\bf [#1]}}
\definecolor{blau}{rgb}{0.1,0.0,0.9}

\newcounter{komcounter}
\numberwithin{komcounter}{section}

\begin{document}

\title{Strong maximum principle and boundary estimates for nonhomogeneous elliptic equations}

\author{
Niklas L.P. Lundstr\"{o}m{\small{$^1$}},
Marcus Olofsson{\small{$^2$}},
Olli Toivanen{\small{$^3$}}
\\\\
{\small{$^1$}}\it \small Department of Mathematics and Mathematical Statistics, Ume{\aa} University,\\
\it \small SE-90187 Ume{\aa}, Sweden\/{\rm ;}
\it \small niklas.lundstrom@umu.se
\\\\
{\small{$^2$}}\it \small Department of Mathematics, Uppsala University,\\
\it \small SE-75236 Uppsala, Sweden\/{\rm ;}
\it \small marcus.olofsson@math.uu.se
\\\\
{\small{$^3$}}\it \small  Department of Applied Physics, University of Eastern Finland,\\
\it \small Yliopistonranta 1, 70210 Kuopio, Finland\/{\rm ;}
\it \small olli@ollitoivanen.info
}

\maketitle


\begin{abstract}

\noindent
We give a simple proof of the strong maximum principle for viscosity subsolutions of
fully nonlinear elliptic PDEs on the form
$$
F(x,u,Du,D^2u) = 0
$$
under suitable structure conditions on the equation allowing for non-Lipschitz growth in the gradient terms.
In case of smooth boundaries, we also prove the Hopf lemma, the boundary Harnack inequality and that positive viscosity solutions vanishing on a portion of the boundary are comparable with the distance function near the boundary.
Our results apply to weak solutions of an eigenvalue problem for the variable exponent $p$-Laplacian. \\


\end{abstract}


\section{Introduction}

\setcounter{theorem}{0}
\setcounter{equation}{0}

We consider fully nonlinear degenerate elliptic equations in nondivergence form
\begin{align}\label{eq:main}
F(x,u,Du,D^2u) = 0. \tag{$PDE$}
\end{align}
Here, $Du$ is the the gradient, $D^2u$ the hessian,
$F:\mathbb{R}^n \times \mathbb{R} \times \mathbb{R}^n  \times \mathcal{S}^n \rightarrow \mathbb{R}$ and $\mathcal{S}^n$ is the set of symmetric $n \times n$ matrices equipped with the positive semi-definite ordering;
for $X, Y \in S^n$, we write $X \leq Y$ if $\langle (X - Y) \xi, \xi \rangle \leq 0$ for all $\xi \in \mathbf{R}^n$.
We assume that $F$ satisfies the degenerate ellipticity
\begin{align}\label{eq:ass_deg_ellipt}
F(x,u,p,X) \geq  F(x,v,p,Y)  \quad \text{whenever} \quad u \geq v  \quad \text{and} \quad X \leq Y, \tag{$F_1$}
\end{align}
together with one of the following structure conditions which will depend on our results.


In order to prove the Strong Maximum Principle (SMaP) in Theorem \ref{thm:SMP} and Hopf lemma (Theorem \ref{corr:hopf}) we assume that
\begin{align}\label{eq:ass_drift_super}
-F(x,0,p,X) \leq \phi(|p|) - \mathcal{P}^-_{\lambda,\Lambda}(X) \quad \text{whenever} \quad x, p \in \mathbb{R}^n, X \in \mathbb{S}^n,
\tag{$F_2 A$}
\end{align}
%
where
\begin{align*}
\mathcal{P}^-_{\lambda,\Lambda}(X) = -\Lambda \text{Tr}(X^+) + \lambda \text{Tr}(X^-)
\end{align*}
is the Pucci maximal operator (see Section \ref{sec:prel} for details),
$X = X^+ - X^-$ with $X^+ \geq 0$, $X^- \geq 0$ and $X^+  X^- = 0$ and $\phi : [0,\infty) \to [0,\infty)$ is a strictly increasing continuous function satisfying $\phi(t) \geq t$ and the Osgood condition
\begin{align}\label{eq:ass_phi_near}
\int_{0}^{1} \frac{dt}{\phi(t)} = \infty.
\tag{$\phi_A$}
\end{align}
Observe that when $\phi(t) > t$ the drift term is not Lipschitz continuous and the equation is nonhomogeneous.
A prototype example is the variable exponent $p$-Laplace equation
\begin{align*}
 - \nabla \cdot \left( |\nabla u|^{p(x)-2} \nabla u\right) = 0,
\end{align*}
where $1 < p^- \leq p(x) \leq p^+< \infty$ is Lipschitz continuous, satisfying our assumptions with $\phi(t) = C \left(|\log t| + 1\right) t$, $\lambda = \min\{ 1, p^{-} - 1\}$ and $\Lambda = \max\{ 1, p^{+} - 1\}$ (see Section \ref{sec:variable-exp}).
A counterexample by Julin \cite{J13}, see Remark \ref{remark:counter1}, shows that \eqref{eq:ass_phi_near} is necessary for SMaP, the maximum principle as well as for the comparison principle and uniqueness.

The above assumptions \eqref{eq:ass_deg_ellipt} and \eqref{eq:ass_drift_super} are implied by e.g. the standard ellipticity assumption that for some $\lambda$, $\Lambda > 0$ we have
\begin{align}\label{eq:ass_ellipt}
\lambda \text{Tr}(Y) \leq F(x,u,p,X) - F(x,u,p,X + Y) \leq \Lambda \text{Tr}(Y). 
\end{align}
whenever $Y$ is positive semi-definite, together with
\begin{align}\label{eq:ass_drift_super_old}
-F(x,0,p,0) \leq \phi(|p|) \quad \text{whenever} \quad x, p \in \mathbb{R}^n.
\end{align}
Let us also observe that assumptions \eqref{eq:ass_deg_ellipt} and \eqref{eq:ass_drift_super} allow nonlinear degenerate
elliptic operators, which do not satisfy \eqref{eq:ass_ellipt}, of the form
$$
F(X) = - \Lambda \left( \sum_{i=1}^{n} \Phi(\mu_i^+)\right) + \lambda \left( \sum_{i=1}^{n} \Psi(\mu_i^-) \right)
$$
where $\mu_i, i = 1,\dots,n$, are the eigenvalues of the matrix $X \in \mathbb{S}^n$ and $\Phi, \Psi : [0,\infty) \to [0,\infty)$
are continuous and nondecreasing functions such that $\Phi(s) \leq s \leq \Psi(s)$, see Capuzzo-Dolcetta--Vitolo \cite{CV07}.

The Strong Minimum Principle (SMiP) for viscosity supersolutions (Remark \ref{re:remark-minimum principles}), as well as some of our boundary estimates, require the symmetric assumption
\begin{align}\label{eq:ass_drift_sub}
F(x,0,p,X) \leq \phi(|p|) + \mathcal{P}^+_{\lambda,\Lambda}(X) \quad \text{whenever} \quad x, p \in \mathbb{R}^n, X \in \mathbb{S}^n,
\tag{$F_2 B$}
\end{align}
in place of \eqref{eq:ass_drift_super} where $\mathcal{P}^+_{\lambda,\Lambda}(X) = -\lambda \text{Tr}(X^+) + \Lambda \text{Tr}(X^-)$ is the Pucci minimal operator.



Strong Maximum Principles for solutions of PDEs have attracted lots of attention during the last decades.
For linear equations the SMaP dates back to Hopf (see Gilbarg--Trudinger \cite{GT} for a proof) 
and Calabi \cite{C58} considered generalized solutions of linear equations in nondivergence form.
For nonlinear problems, mainly modeled on the $p$-Laplacian on divergence form, 
see Pucci--Serrin \cite{PS07}.
Kawohl--Kutev \cite{KK98} and Bardi--Da Lio \cite{BD99, BD01, BD03} considered fully nonlinear degenerate elliptic PDEs, 
and Bardi--Goffi \cite{BG19} nonlinear subelliptic
PDEs modeled on H\"ormander vector fields.
For nonhomogeneous PDEs of variable exponent $p$-Laplace type, a proof was given by
Fan--Zhao--Zhang \cite{FZZ03} while Zhang \cite{Z05} considered a larger class of equations and proved a boundary Hopf point lemma for weak $C^1$-supersolutions. 
See also Wolanski \cite[Theorem 4.1]{W13} for a proof of SMaP for $C^1$-subsolutions.
Capuzzo-Dolcetta--Vitolo \cite{CV07} investigated the validity of the Alexandrov-Bakelman-Pucci maximum principle and of the Phragmen-Lindel\"of principle for fully nonlinear PDEs,
and
Julin \cite{J13} presented a sharp version of the Harnack inequality, quantifying the SMiP, for operators similar to those considered in this paper but without dependence on $u$.

The proof of our SMaP builds mainly on comparison 
with certain classical supersolutions of \eqref{eq:main} constructed with inspiration from Avelin-Julin \cite{AJ17}, see Section \ref{sec:auxiliary}.
%
%
Using similar classical solutions in a barrier argument inspired by Aikawa--Kilpel\"ainen--Shanmugalingam--Zhong~\cite{AKSZ07}, we also prove that in domains with $C^{1,1}$-boundary,
positive viscosity solutions vanishing on a portion of the boundary are comparable with
the distance function near the boundary (Corollary \ref{corr:dist}).
As a consequence, we obtain a boundary Harnack inequality
(Corollary \ref{thm:harnack}) for positive viscosity solutions to \eqref{eq:main}.

More specifically, our boundary estimate consists of a lower estimate (Theorem \ref{le:lower}) and an upper estimate (Theorem \ref{le:upper}) generalizing the main results of Adamowicz--Lundstr\"om \cite{AL16} to viscosity solutions of more general PDEs.
To prove the lower estimate we need a sligthly stronger structural assumption than above in order to build a positive barrier function (Lemma \ref{le:barrier_sub2}) and to this end we assume
\begin{align} \label{eq:ass_drift_sub_ast}
F(x,r,p,X) \leq \phi(|p|)  + \mathcal{P}^+_{\lambda,\Lambda}(X) + \gamma(r) \quad \text{whenever} \quad r \in \mathbb{R}_+, 
x, p \in \mathbb{R}^n, X \in \mathbb{S}^n,
\tag{$F_2 B^\ast$}
\end{align}
in place of \eqref{eq:ass_drift_sub}. Here, $\phi$ is assumed to satisfy \eqref{eq:ass_phi_near} and $\gamma$ is a function dominated by $\phi$ in the sense that $\gamma(r) \leq C^\ast \phi(r)$ whenever $r \leq 1$.
To prove the upper estimate we need to deal with large gradients, forcing us to assume
that for every $\epsilon > 0$  small enough, it holds that
\begin{align}\label{eq:ass_phi_inf_new}
 \int_{0}^{\epsilon} f(t) dt \to \infty  \quad\text{as} \quad \nu \to \infty \tag{$\phi_B$}
\end{align}
where $f(t)$ solves the differential equation $f'(t) = - \phi(f(t))$ with initial condition $f(0) = \nu$.

Condition \eqref{eq:ass_phi_inf_new},
whose necessity will be shown by a counterexample in Remark \ref{re:necessity-phi-inf},
holds, e.g., if there are $s_0$ and $C$ such that
\begin{align}\label{eq:ass_Cs^2}
\phi(s) \leq C s^2, \qquad \text{whenever} \qquad s \geq s_0,
\end{align}
or if the Keller-Osserman type condition,
\begin{align}\label{eq:ass_phi_inf}
\int_{1}^{\infty} \frac{dt}{\phi(t)} = \infty
\end{align}
holds, see Remark \ref{re:assumption-phi-inf}.

The studies of boundary Harnack type inequalities for solutions of differential equations have a long history.
In the setting of harmonic functions on Lipschitz domains, such a result
was proposed by Kemper \cite{K72} and later studied by Ancona \cite{An78}, Dahlberg \cite{D77} and
Wu \cite{W78}. Subsequently, Kemper’s result was extended by Caffarelli--Fabes--Mortola--Salsa \cite{CFMS81} to a class
of elliptic equations, by Jerison--Kenig \cite{JK82} to the setting of nontangentially accessible domains,
Banuelos--Bass--Burdzy \cite{BBB91} and Bass--Burdzy \cite{BB91} considered H\"older domains while Aikawa \cite{Ai01} studied uniform domains.
Recently, Silva--Savin \cite{SS20} gave a short proof of a boundary Harnack inequality for solutions to uniformly elliptic equations in divergence and non-divergence form.
Concerning nonlinear PDEs,
Aikawa--Kilpel\"ainen--Shanmugalingam--Zhong~\cite{AKSZ07} proved similar results in the setting of positive $p$-harmonic functions while in the same year Lewis--Nystr\"om~\cite{LN07, LN10, LN12} started to develop a theory for proving boundary estimates such as the boundary Harnack inequality for $p$-Laplace type equations in more general geometries such as Lipschitz and Reifenberg-flat domains, and
Lundstr\"om \cite{N11, N16} estimated the growth of positive $p$-harmonic functions,
$n-m < p \leq \infty$,
vanishing on an $m$-dimensional hyperplane in $\mathbb{R}^n$, $0 \leq m \leq n-1$.
For nonhomogeneous equations Adamowicz--Lundstr\"om \cite{AL16} proved similar boundary estimates as we do here
but for positive $p(x)$-harmonic functions, and
Avelin--Julin \cite{AJ17} proved a sharp boundary Harnack inequality for operators similar to those considered in this paper but without dependence on $u$.
Concerning applications of boundary Harnack type inequalities we mention free boundary problems and
studies of the Martin boundary, see e.g. \cite{LN10, LN12}.


We end the paper by showing that our results holds for weak solutions of an eigenvalue problem for the variable exponent $p$-Laplacian,
for which we also give a proof of the comparison principle.


\setcounter{theorem}{0}
\setcounter{equation}{0}

\section{Preliminaries}
\label{sec:prel}

By $\Omega$ we denote a domain, that is, an open connected set.
For a set $E \subset \mathbf{R}^n$
we let $\overline{E}$ denote the closure,
$\partial E$ the boundary and
$\complement E$ the complement of $E$. 
Further, $d(x,E)$ denotes the Euclidean distance from $x \in \mathbf{R}^n$ to $E$, and
$B(x,r) = \{ y : | x  -  y | < r \}$
denotes the open ball with radius $r$ and center $x$.
By $c$ we denote a constant $\geq 1$, not necessarily the same at each occurrence.
Moreover, $c(a_1, a_2, \dots, a_k)$ denotes a constant $\geq 1$, not necessarily the same at each occurrence,
depending only on $a_1, a_2, \dots, a_k$.

Let $X \in \mathbf{S}^n$ have eigenvalues $e_1,e_2, \dots ,e_n$.
The Pucci extremal operators $\mathcal{P}^+_{\lambda,\Lambda}$ and $\mathcal{P}^-_{\lambda,\Lambda}$ with ellipticity constants $0 < \lambda \leq \Lambda$ are defined by
$$
\mathcal{P}^+_{\lambda,\Lambda}(X) := - \lambda \sum_{e_i \geq 0} e_i - \Lambda\sum_{e_i < 0} e_i \quad \text{and} \quad \mathcal{P}^-_{\lambda,\Lambda}(X) := - \Lambda \sum_{e_i \geq 0} e_i - \lambda \sum_{e_i < 0} e_i.
$$
For properties of the Pucci operators see e.g. Caffarelli--Cabre \cite{CC95} or Capuzzo-Dolcetta--Vitolo \cite{CV07}.

We next recall the standard definition of  viscosity solutions.
An upper semicontinuous function $u : \Omega \to \mathbb{R}$ is a viscosity subsolution 
if for any $\varphi \in C^2(\Omega)$ and any $x_0 \in \Omega$
such that $u - \varphi$ has a local maximum at $x_0$
it holds that
\begin{align}\label{eq:viscos1}
F(x_0,u(x_0),D\varphi(x_0),D^2\varphi(x_0)) \leq 0.
\end{align}
An lower semicontinuous function $u : \Omega \to \mathbb{R}$ is a viscosity supersolution 
if for any $\varphi \in C^2(\Omega)$ and any $x_0 \in \Omega$
such that $u - \varphi$ has a local minimum at $x_0$
it holds that
\begin{align}\label{eq:viscos2}
F(x_0,u(x_0),D\varphi(x_0),D^2\varphi(x_0)) \geq 0.
\end{align}
A continuous function is a viscosity solution if it is both a viscosity sub- and viscosity supersolution.
In the following we sometimes drop the word ``viscosity" and simply write subsolution, solution and supersolution.
We say that a subsolution (supersolution) is strict in a domain $\Omega$ if the equality in \eqref{eq:viscos1} (\eqref{eq:viscos2}) is strict,
and we say that a subsolution, supersolution or solution is classical if it is twice differentiable in $\Omega$.

Let $u$ be a subsolution and $v$ a supersolution to \eqref{eq:main} and let $a$ and $b$ be constants.
As \eqref{eq:main} is not homogeneous,  $a + b u$ and $a + b v$ are not necessarily sub- and supersolutions.
However, degenerate ellipticity \eqref{eq:ass_deg_ellipt} guarantees that $u - c$ is a subsolution
and $v + c$ is a supersolution whenever $c \geq 0$ is a constant.

We will not discuss the validity of a general comparison principle for viscosity solutions of \eqref{eq:main} since our
proofs only require the comparison of viscosity solutions to \textit{classical} strict sub- and supersolutions, a comparison which is always possible.
Indeed, let $\Omega$ be a bounded domain,
$u$ a viscosity subsolution, and $v$ a classical strict supersolution in $\Omega$. Assume that
$u \leq v$ on $\partial \Omega$ and that $u \geq v$ somewhere in $\Omega$.
By USC the function $u-v$ attains a maximum at some point $x_0 \in \Omega$.
Since $v \in C^2(\Omega)$, $u - v$ has a maximum at $x_0$ and $u$ is a viscosity subsolution it follows by definition of viscosity solutions that
\begin{align}\label{eq:johej-lemma1}
F(x_0, u(x_0), Dv(x_0), D^2 v(x_0)) \leq 0.
\end{align}
But since $v$ is a classical strict supersolution we have $F(x,v(x), Dv(x), D^2v(x)) > 0$ whenever $x \in \Omega$,
and as $u(x_0) \geq v(x_0)$ it follows from \eqref{eq:ass_deg_ellipt} that
%
$
F(x_0, u(x_0), Dv(x_0), D^2 v(x_0)) \geq F(x_0, v(x_0), Dv(x_0), D^2 v(x_0)) > 0.
$
%
This contradicts \eqref{eq:johej-lemma1} and hence we have proved the following simple lemma.

\begin{lemma}\label{le:comp-weak}
Let $\Omega$ be a bounded domain, $u \in USC(\overline{\Omega})$ a viscosity subsolution and $v \in LSC(\overline{\Omega})$ a viscosity supersolution to \eqref{eq:main}, satisfying $u \leq v$ on $\partial \Omega$.
Assume \eqref{eq:ass_deg_ellipt}.
If either $u$ is a strict classical subsolution, or $v$ is a strict classical supersolution,
then $u < v$ in $\Omega$.
\end{lemma}

%


When dealing with boundary estimates we make use of the following ball conditions.
A point $w \in \partial \Omega$, where $\Omega\subset \mathbb{R}^n$ is a domain,
satisfies the \emph{interior ball condition} with radius $r_i > 0$ if
there exists $\eta^i \in \Omega$ such that $B(\eta^i, r_i) \subset \Omega$ and
$\partial B(\eta^i, r_i) \cap \partial \Omega = \{w\}$.
Similarly, $w\in \partial \Omega$ satisfies the \emph{exterior ball condition}
with radius $r_e > 0$ if there exists $\eta^e \in \mathbb{R}^n \setminus \Omega$
such that $B(\eta^e, r_e) \subset \mathbb{R}^n \setminus \Omega$ and $\partial B(\eta^e, r_e) \cap \partial \Omega = \{w\}$.
A point $w \in \partial \Omega$ satisfies the \emph{ball condition} with radius $r_b$
if it satisfies both the interior ball condition and the exterior ball condition with radius $r_b$.
A domain $\Omega\subset \mathbb{R}^n$ is said to satisfy the \emph{ball condition (or interior or exterior ball condition)},
if such condition holds for all $w\in \partial \Omega$.

It is well known that $\Omega\subset \mathbb{R}^n$ satisfies the ball condition if and only if $\Omega$ is a $C^{1,1}$-domain,
see Aikawa--Kilpel\"ainen--Shanmugalingam--Zhong \cite[Lemma 2.2]{AKSZ07} for a proof.


\setcounter{theorem}{0}
\setcounter{equation}{0}

\section{Construction of auxiliary functions}
\label{sec:auxiliary}

This section is devoted to the construction of some classical strict sub- and supersolutions to \eqref{eq:main}
which will be used in our proofs of the maximum principles and the boundary estimates.
The first lemma yields negative subsolutions and positive supersolutions with arbitrary small gradients in an annulus,
needed in the proof of Hopf lemma, the SMaP and SMiP.
\begin{lemma}
\label{le:barrier_sub}
Assume that \eqref{eq:ass_deg_ellipt}, \eqref{eq:ass_drift_sub} and \eqref{eq:ass_phi_near} hold.
Suppose that $r \leq r^\ast$, $0 < M$ and $y \in \mathbf{R}^n$ are given and let $A = B(y,2r) \setminus \overline{B(y, r)}$.
Then there exist positive $m(\lambda, \Lambda, n,\phi, r^\ast) \leq {M}$, $\mu(\lambda, \Lambda, n, \phi, r^\ast, M)$ and a strict classical subsolution $\check{u}$ to \eqref{eq:main} in $A$ such that
\begin{align*}
\check{u} = 0 \;\text{on}\;\; \partial B(y, r), \quad  \check{u} = -mr \;\text{on}\;\;\partial B(y, 2r) \quad \text{and}\quad \mu \leq |D \check{u}(x)| \leq \mu^{-1} \quad \text{in}\;\; A.
\end{align*}
Moreover, if we assume \eqref{eq:ass_drift_super} in place of \eqref{eq:ass_drift_sub},
then $\hat{v} = -\check{u}$ is a strict classical supersolution to \eqref{eq:main} in  $A$.
\end{lemma}


In the above lemma, it is crucial that $M>0$ is arbitrary and that $0< m \leq M$, meaning that we can construct ``arbitrary flat" strict sub- and supersolutions.
Together with \eqref{eq:ass_deg_ellipt} this
allows us to find a strict classical subsolution $\check{u}$ and supersolution $\hat{v}$ in $A$ satisfying
\begin{align}\label{eq:remark1}
\hat{v} = - \check{u} = \tilde M - mr > 0 \quad \text{on}\;\; \partial B(y, r) \quad \mbox{and} \quad \hat{v} = - \check{u} = \tilde M \quad \text{on}\;\;\partial B(y, 2r)
\end{align}
whenever $\tilde M > 0$ is given.


In order to prove the lower boundary growth estimate for positive solutions to \eqref{eq:main} we need a positive barrier from below
in terms of a strict classical subsolution.
The following lemma, which is similar to Lemma \ref{le:barrier_sub}, shows that this is possible.
Note however that,
contrary to Lemma \ref{le:barrier_sub}, we now need to impose some extra control on the growth of $F(x,r,p,X)$ in the parameter $r$.
Namely, we need \eqref{eq:ass_drift_sub_ast} in place of \eqref{eq:ass_drift_sub}.
Due to the generality of the function $\gamma$ in \eqref{eq:ass_drift_sub_ast} we need to shrink the size of the annulus in order to handle the lower order terms in $\gamma(u)$.

\begin{lemma}
\label{le:barrier_sub2}
Assume that \eqref{eq:ass_deg_ellipt}, \eqref{eq:ass_drift_sub_ast} and \eqref{eq:ass_phi_near} hold.
Suppose that $r\leq r^\ast$, $0 < M$ and $y \in \mathbf{R}^n$ are given and let $A = B(y,2r) \setminus \overline{B(y, r)}$.
Then there exist positive $m(\lambda, \Lambda, n,\phi, r^\ast, C^\ast) \leq M$, $\mu(\lambda, \Lambda, n, \phi, r^\ast, C^\ast, M)$
and a strict classical subsolution $\doublecheck{u}$ to \eqref{eq:main} in $A$ such that
\begin{align*}
\doublecheck{u} = m r \;\text{on}\;\; \partial B(y, r), \quad  \doublecheck{u} = 0 \;\text{on}\;\;\partial B(y, 2r) \quad \text{and}\quad \mu \leq |D \doublecheck{u}(x)| \leq \mu^{-1} \quad \text{in}\;\; A.
\end{align*}
\end{lemma}

Next lemma yields a strict classical supersolution that is zero on the inner ball of the annulus
and as large as we need on the outer ball.
The proof therefore needs \eqref{eq:ass_phi_inf}, which is an assumption on $\phi(t)$ for large $t$, to handle the large gradients.

%
\begin{lemma}
\label{le:barrier_super}
Assume that \eqref{eq:ass_deg_ellipt}, \eqref{eq:ass_drift_super} and \eqref{eq:ass_phi_inf_new} hold.
Suppose that $r\leq r^\ast$, 
$0 < M$ and $y \in \mathbf{R}^n$ are given.
Then there exist positive $m(\lambda, \Lambda, n, \phi, r^\ast) \geq {M}$, $\nu(\lambda, \Lambda, n, \phi, r^\ast, M)$, $k(\lambda, \Lambda, n,\phi, r^\ast, M) > 1$ and a strict classical supersolution $\doublehat{v}$ to \eqref{eq:main} in $A_k = B(y,kr) \setminus \overline{B(y, r)}$ such that
\begin{align*}
\doublehat{v} = 0 \;\text{on}\;\; \partial B(y, r), \quad  \doublecheck{v} = m r \;\text{on}\;\;\partial B(y, kr) \quad \text{and}\quad \nu^{-1}\leq |D \doublehat{v}(x)| \leq \nu \quad \text{in}\;\; A.
\end{align*}
\end{lemma}


\noindent
{\bf Proof of Lemma \ref{le:barrier_sub}.}
Let $g(t)$ be a solution to $g'(t) =  C \phi(g(t))$ with initial condition $g(0) = \mu$,
i.e.,
\begin{align}\label{eq:def_of_g}
t = \int_{\mu}^{g(t)} \frac{ds}{C \phi(s)}, \quad \text{whenever} \quad t \in [0,1],
\end{align}
where $\mu$ and $C$ are constants to be chosen later, see Figure \ref{fig:functions}.
By assumption \eqref{eq:ass_phi_near}
the increasing solution $g : [0,1] \rightarrow [\mu, \infty)$ is well defined through the implicit function theorem whenever $\mu = \mu(C,\phi)$ is small enough.
Define
\begin{align}\label{def:smallm}
m(\mu, C,\phi) : =  \int_{0}^{1} g(t) dt 
\end{align}
and
\begin{align*}
\check u(x) =  r \left (\int_{0}^{2 - |x - y|/r} g(t) dt - m\right), \quad \text{whenever}\quad x \in A. 
\end{align*}

By construction $\check{u}(x)< 0$ whenever $x\in A$, $\check u = 0 $ on $\partial B(y, r)$, and $\check{u} =-mr$ on $\partial B(y, 2r)$ for $m$ as defined in \eqref{def:smallm}.
By assumption \eqref{eq:ass_phi_near} we also have that $\mu \rightarrow 0$ implies $g(t) \rightarrow 0$, for all $t \in [0,1]$,
and so $m,\check u \rightarrow 0$ as $\mu \to 0$.
Therefore, by decreasing $\mu$ if necessary, we have $m \leq M$. Our choice of $\mu$ depends only on $C, \phi$ and $M$.

To prove that $\check u$ is a strict classical subsolution to the equation \eqref{eq:main}
we first calculate the derivatives of $\check u$.
For notational simplicity we set
$\# := 2 - \frac{|x - y|}{r}$ 
and obtain
\begin{align*}
\frac{\partial \check u}{\partial x_i} =  -  \frac{x_i - y_i}{|x - y|} g\left(\#\right)
\quad \textrm{giving} \quad
\vert D\check u \vert &= g\left(\#\right).
\end{align*}
The construction of $g(t)$ implies
$
g'(t) = C \phi(g(t))
$
and hence
\begin{align*}
\frac{\partial^2 \check u}{\partial x_j \partial x_i}
=  \frac{1}{r}\frac{(x_i - y_i)(x_j - y_j)}{|x - y|^2}
C \phi\left(g\left(\#\right)\right)
- \left( \frac{\delta_{ij}}{|x - y|} - \frac{(x_i - y_i)(x_j - y_j)}{|x - y|^3} \right)
g\left(\#\right),
\end{align*}
implying
\begin{align*}
\text{Tr}(D^2 \check u) = \frac{1}{r} C  \phi\left(g\left(\#\right)\right) - \frac{n - 1}{|x-y|} g\left(\#\right).
\end{align*}
We decompose $D^2\check u = \left(D^2\check u\right)^+ - \left(D^2\check u\right)^-$ so that
(any $X \in \mathbb{S}^n$ can be decomposed as $X = X^+ - X^-$ with $X^+ \geq 0$, $X^- \geq 0$ and $X^+  X^- = 0$)
%
\begin{align*}
\text{Tr}\left( \left(D^2\check u\right)^+ \right) =  \frac{1}{r} C \phi\left(g\left(\#\right)\right) \quad \textrm{and} \quad \text{Tr}\left( \left(D^2\check u\right)^- \right) =  \frac{n - 1}{|x-y|} g\left(\#\right).
\end{align*}
%
Utilizing the structure assumption \eqref{eq:ass_drift_sub} gives, since $\check u \leq 0$, that
\begin{align*}
&F(x,\check u,D\check u,D^2\check u) \leq F(x,0,D\check u,D^2\check u) \leq \phi\left(g\left(\#\right)\right)
+ \Lambda  \frac{n - 1}{|x-y|} g\left(\#\right)
-  \frac{\lambda}{r} C  \phi\left(g\left(\#\right)\right).
\end{align*}
By recalling $t \leq \phi(t)$ and observing that $r \leq |x-y| \leq 2r$ for all $x \in A$ we conclude
\begin{align*}
&F(x,\check u,D\check u,D^2\check u) \leq \phi\left(g\left(\#\right)\right) \left(1 + \Lambda \frac{n - 1}{r} - C \frac{\lambda}{r} \right)
\end{align*}
and thus by taking $C$ large enough, depending only on $\lambda, \Lambda, n$, and $r^\ast$, we obtain
\begin{align*}
F(x,\check u,D\check u,D^2\check u) < 0 \quad \text{in} \quad A.
\end{align*}
Consequently, $\check{u}$ is a strict classical subsolution to \eqref{eq:main} in the annulus $A$.
Finally, after decreasing $\mu$ if necessary, we have
$$
\mu \leq \vert D\check u \vert \leq \mu^{-1}
$$
where $\mu$ depends only on $\lambda, \Lambda, n, \phi, r^\ast$, and $M$.
This completes the proof of the strict classical subsolution $\check{u}$.

By defining $\hat{v} = - \check u$ and mimicking the above reasoning, using \eqref{eq:ass_drift_super} in place of
\eqref{eq:ass_drift_sub}, we conclude that $\hat{v}$ is a strict classical supersolution. The proof of Lemma \ref{le:barrier_sub} is complete.
$\hfill \Box$ \\


\noindent
{\bf Proof of Lemma \ref{le:barrier_sub2}.}
As the proof largely follows that of Lemma \ref{le:barrier_sub} we only lay out the main differences. With $g(t)$ as in \eqref{eq:def_of_g} and $m$ as in \eqref{def:smallm}, define
%
\begin{align*}
\doublecheck u(x) =  r \left (\int_{0}^{2 - |x - y|/r} g(t) dt \right) \quad \text{whenever}\; x \in A. 
\end{align*}
%
Then $\doublecheck{u}$ satisfies the specified boundary conditions while having the same derivatives as the subsolution candidate $\check u$ in Lemma \ref{le:barrier_sub}.
Applying \eqref{eq:ass_drift_sub_ast} we conclude
\begin{align*}
F(x,\doublecheck u,D\doublecheck u,D^2\doublecheck u) &\leq \phi\left(g\left(\#\right)\right)
+ \Lambda  \frac{n - 1}{|x-y|} g\left(\#\right)
-  \frac{\lambda}{r} C  \phi\left(g\left(\#\right)\right) + \gamma( \doublecheck u(x)) \\
&\leq \phi\left(g\left(\#\right)\right) \left(1 + \Lambda  \frac{n - 1}{r} -  \frac{\lambda C}{2 r} \right)
+ \gamma( \doublecheck u(x)) - \frac{\lambda C}{2 r} \phi\left(g\left(\#\right)\right).
\end{align*}
Pick $C = C(\lambda, \Lambda, n, r^\ast, C^\ast)$ so large that the first term on the right hand side of the above inequality is negative
and such that $\lambda C / (2r^\ast) \geq C^\ast$.
It then remains to ensure that 
%
\begin{align}\label{eq:tjohejlower}
\gamma( \doublecheck u(x)) - C^\ast \phi\left(g\left(\#\right)\right) \leq 0
\end{align}
whenever $\doublecheck u(x)$ has a small gradient.
Since $g(t)$ is a strictly increasing function it follows that $\doublecheck u$ satisfies the sub-mean value property and thus
$$
\doublecheck u(x) \leq r |D\doublecheck u(x)| = r g(\#) \to 0 \quad \text{as} \quad \mu \to 0.
$$
%
By assumption $\gamma(s) \leq C^\ast \phi(s)$ whenever $s \leq 1$ and thus \eqref{eq:tjohejlower} holds as long as $r \leq 1$.
By decreasing $\mu$ if necessary, this completes the proof of Lemma \ref{le:barrier_sub2}.
$\hfill \Box$ \\


\noindent
{\bf Proof of Lemma \ref{le:barrier_super}.}
%
Let $f(t)$ be a solution to $f'(t) = - C \phi(f(t))$ with initial condition $f(0) = \nu$,
where $\nu$ and $C$ are constants to be chosen later, see Figure \ref{fig:functions}.
Existence follows in a neighbourhood of $(0, \nu)$ since $\phi$ is continuous.
For small enough $k = k(C, \phi)$, independent of $\nu$ and $r$, we can define
\begin{align}\label{eq:second-super}
\doublehat v(x) =  r \int_{0}^{|x - y| / r - 1} f(t) dt, \quad \text{whenever}\quad x \in A_k = B(y,kr)\setminus \overline{B(y,r)}.
\end{align}
By construction we obtain $\doublehat v = 0$ on $\partial B(y, r)$ and $\doublehat v = m r$ on $\partial B(y, kr)$ where
$$
m = \int_0^k f(t) dt.
$$
Assumption \eqref{eq:ass_phi_inf_new} implies $m \to \infty$ as $\nu \to \infty$ and therefore we can make
$\doublehat v$ so large that  $m > M$ by increasing $\nu$, depending only on $\phi, k, C$ and $M$.

To prove that $\doublehat v$ is a strict classical supersolution to \eqref{eq:main},
put $\Xi = |x - y| / r - 1$ and differentiate to obtain
\begin{align*}
\text{Tr}(D^2 \doublehat v) = -\frac{1}{r} C  \phi\left(f\left(\Xi\right)\right) +  \frac{n - 1}{|x-y|} f\left(\Xi\right) \quad \textrm{and} \quad
\vert D\doublehat v \vert = f\left(\Xi\right).
\end{align*}
Following a similar decomposition of $D^2 \doublehat v$ as in the proof of Lemma \ref{le:barrier_sub}, using $\doublehat v \geq 0$, \eqref{eq:ass_deg_ellipt} and \eqref{eq:ass_drift_super} yield
\begin{align}\label{eq:super_proof2}
&F(x,\doublehat v,D\doublehat v,D^2\doublehat v) \geq  - \phi\left(f\left(\Xi\right)\right)
- \Lambda \frac{n - 1}{|x-y|} f\left(\Xi\right)
+  \frac{\lambda}{r} C  \phi\left(f\left(\Xi\right)\right).
\end{align}
By recalling $t \leq \phi(t)$ and observing that $r \leq |x-y| \leq kr$ we conclude
\begin{align*}
&F(x,\doublehat v,D\doublehat v,D^2\doublehat v) \geq
\phi\left(f\left(\Xi\right)\right) \left(- 1 -  \Lambda \frac{n - 1}{r} + C \frac{\lambda}{r} \right) > 0
\end{align*}
whenever $C$ is large enough, depending only on $\lambda, \Lambda, n,$ and $r^\ast$.
Thus $\doublehat v$ is a strict classical supersolution to \eqref{eq:main} in the annulus.
Moreover, $\doublehat v$ has bounded gradient since, after increasing $\nu$ if necessary,
$$
\nu^{-1} \leq \vert D\doublehat v \vert \leq \nu
$$
where $\nu = \nu (\lambda, \Lambda, n, \phi, r^\ast, M)$.
$\hfill \Box$ \\


\setcounter{theorem}{0}
\setcounter{equation}{0}

\section{Strong maximum principles}
\label{sec:maximum}

We begin by proving the SMaP using a contradiction argument based on comparison with the strict classical supersolution in Lemma \ref{le:barrier_sub}.

\begin{theorem}[Strong Maximum Principle]\label{thm:SMP}
Assume that \eqref{eq:ass_deg_ellipt}, \eqref{eq:ass_drift_super} and \eqref{eq:ass_phi_near} hold.
Let $\Omega \subset \mathbf{R}^n$ be an open connected set and suppose that $u$
is a viscosity subsolution to \eqref{eq:main} in $\Omega$.
If $u$ attains a positive maximum in $\Omega$, then $u$ is constant.
\end{theorem}

\noindent
\begin{proof}
Assume, by contradiction, that a nonconstant USC subsolution $u$ attains its strictly positive maximum $\tilde M$ in $\Omega$ and let $K = \{x \in \Omega \,|\, u(x) = \tilde M\}$.
By assumption $K \neq \Omega, K\neq \emptyset$ and therefore, by USC, there exist $\bar{x} \in \Omega \cap \partial K$ and $s_\ast > 0$ such that for every $s <s_\ast$ there is $y_s \in \Omega$ such that
$B(y_s,s) \in \Omega \setminus K$, $\bar{x} \in \partial B(y_s,s)$ and $u(x) < \tilde M$ in $B(y_s,s)$.

For any $y_s \in \Omega$ and for arbitrary positive $M$ and $r \in (0, r_*]$, Lemma \ref{le:barrier_sub} yields existence of a classical strict supersolution $\hat{v}$ in the annulus $A = B(y_s,2r)\setminus \overline{B(y_s,r)}$. 
Moreover, since $\tilde{M} > 0$ we see from \eqref{eq:remark1} that we can chose $\hat v$ to satisfy
\begin{align*}
\hat{v} = \tilde M - mr   \;\text{on}\;\; \partial B(y_s, r), \quad \mbox{and} \quad  \hat{v} = \tilde M \;\text{on}\;\;\partial B(y_r, 2r),
\end{align*}
with constant $m$ satisfying $\tilde M - mr \geq 0$ for all $r \in (0,r_\ast]$.

Choose $r$ so that $r < s < 2r$ and $B(y_s,2r) \in \Omega$ (after decreasing $s$ if necessary).
It follows that $u \leq \hat{v} =  \tilde M$ on $\partial B(y_s,2r)$.
Using the possibility of decreasing $m$ once again and the fact that $u < \tilde{M}$ on $\partial B(y,r)$,
we can take $m$ so that $u \leq \hat{v}$ on $\partial B(y,r)$ as well.
Lemma \ref{le:comp-weak} therefore implies $u < \hat{v}$ in the annulus $A = B(y_s,2r)\setminus \overline{B(y_s,r)}$.
As the gradient of $\hat{v}$ does not vanish in $A$, $\bar{x} \in A$ and $u(\bar{x}) = \tilde{M}$, we arrive at a contradiction.
\end{proof}\\


Using a similar argument we next prove the following version of Hopf lemma.

%
%

\begin{theorem}[Hopf lemma]\label{corr:hopf}
Assume that \eqref{eq:ass_deg_ellipt}, \eqref{eq:ass_drift_super} and \eqref{eq:ass_phi_near} hold.
Let $\Omega \subset \mathbf{R}^n$ be a domain,
$w\in \partial \Omega$,
and suppose that $\Omega$ satisfies an interior ball condition with $B(y_w, r_w)$ at $w$.
If $u$ is a viscosity subsolution to \eqref{eq:main} in $\Omega$,
upper semicontinuous on $\Omega \cup \{w\}$ so that $u(w) > \limsup_{\Omega \ni x \to w} u(x)$ whenever $x \in \Omega$, and $u(w) > 0$,
then, for any $v\in \mathbf R^n$ such that $ v \cdot (w - y) < 0$ we have
\begin{align*}
\limsup_{s \to 0} \frac{u(w + s v) - u(w)}{s} < 0.
\end{align*}
\end{theorem}

\noindent
\begin{proof}
Choose $r < r_w/2$ and let $y_r= w + 2r \frac{(y_w-w)}{|(y_w-w)|}$.
Put $\tilde{M} = v(w)<0$.
Following the argument in the proof of Theorem \ref{thm:SMP} yields a strict classical supersolution $\hat u$
in the annulus $A = B(y_r,2r)\setminus \overline{B(y_r,r)}$ satisfying
$\hat{v} = \tilde M$ on $\partial B(y_r, 2r)$ and $\hat{v} = \tilde M - mr$ on $\partial B(y_r, r)$.
Choosing $r$ small enough, we can also ensure that $\hat u > u$ on $B(y_r,r)$ and thus,
from the weak comparison principle we have that $\hat{u} > u$ in $A$.
Hence, since $\hat{u}$ has bounded gradient
$$
\limsup_{s \to 0} \frac{u(w + s \gamma) - u(w)}{s} \leq D\check{u} \cdot \gamma < 0.
$$
\end{proof}\\


\begin{remark}\label{remark:counter1}
Assumption \eqref{eq:ass_phi_near} is necessary for Theorems \ref{corr:hopf} and \ref{thm:SMP} to hold.
\end{remark}
Indeed, assume that \eqref{eq:ass_phi_near} does not hold and
follow Julin \cite{J13} by defining $h : (-1, 1) \to [0, \infty)$ so that $h(x) = 0$ for $(-1, 0)$ and
\begin{align*}
x = \int_{0}^{h(x)} \frac{dt}{\phi(t)} \quad \text{for} \; x \in [0,1).
\end{align*}
Then $h' = \phi(v)$ and the function
\begin{align*}
H(x) = 1 - \int_{0}^{x} h(s) ds
\end{align*}
satisfies $H'' + \phi(|H'|) = 0$ classically on $(-1,1)$ and violates both the Hopf lemma and the SMaP
(SMiP by $-H$).
Extending $H$ to $(-3,1)$ evenly in the line $x = -1$ produces a classical solution contradicting the maximum principle, the comparison principle and uniqueness as well, since constants solves the equation, see Figure \ref{fig:functions}.

We end this section by noting that by using the auxiliary subsolution from Lemma \ref{le:barrier_sub} in place of the supersolution in similar arguments as in the above proofs we conclude the following versions of Theorem \ref{corr:hopf} and Theorem \ref{thm:SMP} for viscosity supersolutions.

\begin{remark}\label{re:remark-minimum principles}
Under the assumptions in Theorem \ref{corr:hopf} but with \eqref{eq:ass_drift_super} replaced by \eqref{eq:ass_drift_sub},
if $v$ is a viscosity supersolution to \eqref{eq:main} in $\Omega$, lower semicontinuous on $\Omega \cup \{w\}$ such that $v(w) < \liminf_{\Omega \ni x \to w} v(x)$ and $v(w) < 0$, 
then, for any $\gamma \in \mathbf R^n$ such that $ \gamma \cdot (w - y_w) < 0$ we have
\begin{align*}
\liminf_{s \to 0} \frac{v(w + s \gamma) - v(w)}{s} > 0.
\end{align*}
Moreover, under the assumptions in Theorem \ref{thm:SMP} but with \eqref{eq:ass_drift_super} replaced by \eqref{eq:ass_drift_sub}
a viscosity supersolution to \eqref{eq:main} may not attain a negative minimum in $\Omega$ unless it is constant.
\end{remark}


\setcounter{theorem}{0}
\setcounter{equation}{0}

\section{Boundary growth estimates}
\label{sec:boundary}

We begin by proving that a positive supersolution vanishing on a boundary satisfying an interior ball condition may not vanish faster than the distance to the boundary.

\begin{theorem}[Lower estimate]\label{le:lower}
Assume that \eqref{eq:ass_deg_ellipt}, \eqref{eq:ass_drift_sub_ast} and \eqref{eq:ass_phi_near} hold.
Let $\Omega \subset \mathbf{R}^n$ be a domain satisfying the interior ball
condition with radius $r_i$ and take $w \in \partial \Omega$ and $r$ s.t. $0<2r < r_i$.
Assume that $v$ is a positive viscosity supersolution to \eqref{eq:main} in
$\Omega \cap B(w, 6r)$ satisfying $v = 0$ on $\partial \Omega \cap B(w,6r)$.
Then there exists a constant $c = c(\lambda, \Lambda, n, \phi, r_i, r^{-1}{\inf_{\Gamma_{w,r}} v}, C^\ast)$ such that
\begin{align*}
c\, v(x) \geq  d(x, \partial\Omega) \qquad \text{whenever} \quad x \in \Omega \cap B(w, r),
\end{align*}
where $\Gamma_{w,r} = \{x \in \Omega | r < d(x,\partial \Omega) < 3r \} \cap B(w,6r)$.
\end{theorem}


\noindent
\begin{proof}
We follow the lines of Aikawa--Kilpel\"ainen--Shanmugalingam--Zhong \cite[Lemma 3.1]{AKSZ07} and
take $x \in \Omega \cap B(w, r)$ and let $\eta \in \partial \Omega$ be such that $d(x, \partial\Omega) = |x - \eta|$.
By the interior ball condition at $\eta$ we can find a point
$\eta^i$ such that $B(\eta^i, r_i) \subset \Omega$ and $\eta \in \partial B(\eta^i, r_i)$.
Now, take the point $\eta^i_{2r}$ which is such that $\eta= \eta^i_{2r} +2r \frac{\eta-\eta_i}{|(\eta-\eta_i)|}$, i.e., $\eta^i_{2r} \in \Omega$ lies on a straight line $\gamma$ between $\eta$ and $\eta^i$ on a distance $2r$ from the boundary $\partial \Omega$. Then $B(\eta^i_r, 2 r)\subset \Omega$ so $u$ is a positive viscosity supersolution in $B(\eta^i_{2r}, 2r)$.  By construction $|x-\eta| < r$ so $|\eta^i_{2r}-x| > r$ and by the interior ball condition we get that $x$ lies on the line $\gamma$ and thus $x \in A =B(\eta^i_{2r},2r)\setminus \overline{	B(\eta^i_{2r}, r)}$.

Next we note that $B(\eta^i_r, r) \subset \Gamma_{w,r}$ and that, for $m > 0$ small enough, we have
\begin{align*}
v(x) \geq \inf_{\Gamma_{w,r}} v \geq m r > 0 \qquad \hbox{whenever} \quad x \in \Gamma_{w,r}.
\end{align*}
Now apply Lemma \ref{le:barrier_sub2}	with $r^\ast=r_i/2$ and $M=r^{-1}\inf_{\Gamma_{w,r}} v$  to ensure existence of a classical strict subsolution $\doublecheck{u}$ in the annulus $A \subset \Omega \cap B(w,6r)$. Moreover, this subsolution can be taken to satisfy $\doublecheck{u} =0$ on $\partial B(\eta_r^i, 2r)$ and $\doublecheck{u} =mr \leq v(x)$ on $\partial B(\eta_r^i, r)$, for some $m = m(\lambda, \Lambda, n, \phi,r_i) \in (0, {r^{-1}\inf_{\Gamma_{w,r}} u})$. Since $\doublecheck u \leq v$ on the boundaries of the annulus $A= B(\eta_r^i, 2r) \setminus \overline{B(\eta_r^i, r)}$ and $\doublecheck{u}$ is a classical strict subsolution and
we get from the weak comparison principle in Lemma \ref{le:comp-weak} that $\doublecheck u(x) \leq v(x)$ for all $x\in A$.

Since $x \in A$ the result will follow after ensuring that $\doublecheck u$
does not vanish faster than $d(x, \partial\Omega)$ as $x \to \partial\Omega$.
However, this is a consequence of Lemma \ref{le:barrier_sub2} which gives the existence of a constant
$c = c(\lambda, \Lambda, n, \phi, r^{-1}\inf_{\Gamma_{w,r}} v, C^\ast)$, independent of $x$, such that
$c^{-1} \leq |D\doublecheck{u}| \leq c$. The proof is thus complete.
\end{proof}\\


With the positive subsolution barriers constructed in Lemma \ref{le:barrier_sub2} we may repeat the proof of the Hopf lemma for viscosity supersolutions in Remark \ref{re:remark-minimum principles} to relax the strict inequality $v(\omega) < 0$ at the boundary to $v(\omega)\leq 0$.
This slightly stronger result and the SMiP imply the following generalization of 
Corollary 5.2 in Adamowicz--Lundstr\"om \cite{AL16} and Proposition 6.1 in Aikawa--Kilpel\"ainen--Shanmugalingam--Zhong \cite{AKSZ07}:

\begin{corollary}
Suppose that \eqref{eq:ass_deg_ellipt}, \eqref{eq:ass_drift_sub_ast}, \eqref{eq:ass_phi_near} hold,
and that $v \geq 0$ is a positive viscosity supersolution to \eqref{eq:main} in $\Omega$.
If there exists a point $w \in \partial \Omega$ satisfying an interior ball condition such that
\begin{align}\label{eq:corr-version}
\liminf_{\Omega \ni y \to w} \frac{v(y)}{d(y,\partial \Omega)} = 0,
\end{align}
then $v\equiv 0$ in $\Omega$.
\end{corollary}
\begin{proof}
 As the above limit is zero and $v \geq 0$,
the viscosity supersolution $v$ must vanish at a minimum at the boundary point $w \in \partial \Omega$.
If this minimum is strict relative to the interior of $\Omega$,
then the Hopf lemma in Remark \ref{re:remark-minimum principles} violates \eqref{eq:corr-version}.
Therefore, we must have $v(x) = 0$ in some small interior ball at $w$ and $v \equiv 0$ then follows by Remark \ref{re:remark-minimum principles}.
\end{proof}\\


We next prove an upper estimate of viscosity subsolutions vanishing on a portion of a domain satisfying an exterior ball condition.

\begin{theorem}[Upper estimate]\label{le:upper}
Assume that \eqref{eq:ass_deg_ellipt}, \eqref{eq:ass_drift_super} and \eqref{eq:ass_phi_inf_new} hold.
Let $\Omega \subset \mathbf{R}^n$ be a domain satisfying the exterior ball
condition with radius $r_e$ and take $w \in \partial \Omega$ and r s.t. $0 < 2r < r_e$.
Assume that $u$ is a strictly positive viscosity subsolution to \eqref{eq:main} in
$\Omega \cap B(w, 6r)$ satisfying $u = 0$ on $\partial \Omega \cap B(w,6r)$.
Then there exists a constant $c = c(\lambda, \Lambda, n, \phi, r_e, r^{-1}\sup_{B(w,6 r)\cap \Omega} u)$ such that
\begin{align*}
 u(x) \leq c \, d(x, \partial\Omega) \qquad \text{for} \quad x \in \Omega \cap B(w, r).
\end{align*}
\end{theorem}

\noindent
\begin{proof}
We follow the lines of Aikawa--Kilpel\"ainen--Shanmugalingam--Zhong \cite[Lemma 3.3]{AKSZ07} and
take $x \in \Omega \cap B(w, r)$ and let $\eta \in \partial \Omega$ be such that $d(x, \partial\Omega) = |x - \eta|$.
By the exterior ball condition at $\eta$ we can find a point $\eta^e$ such that $B(\eta^e, r_e) \subset\mathbf{R}^n \setminus\Omega$
and $\eta \in \partial B(\eta^e, r_e)$.
Now, take the point $\eta^e_{r}$ which is such that $\eta= \eta^e_{r} + r \frac{\eta-\eta^e}{|(\eta-\eta^e)|}$ (i.e., $\eta^e_r \not \in \Omega$ lies on the straight line $\gamma$ between $\eta$ and $\eta^e$ on a distance $r$ from the boundary $\partial \Omega$)

Applying Lemma \ref{le:barrier_super} we ensure existence of a classical strict supersolution $\doublehat v$ in the annulus $A = B(\eta^e_r,2r)\setminus \overline{B(\eta^e_r,r)}$
satisfying $\hat{v} = 0$ on $\partial B(y, r)$ and $\hat{v} = mr$ on $\partial B(y, 2r)$. Moreover, we are free to choose $m$ such that $m \geq M = r^{-1}\sup_{B(w,6r)\cap \Omega} u$ so that $\doublehat v > u$ on the boundary of $B(\eta^e_r,2r) \cap \Omega$ (since $\doublehat v \geq 0$ and $u$ vanishes on $\Omega \cap B(w,6r)$).  By the weak comparison principle Lemma~\ref{le:comp-weak}, we obtain therefore $u \leq \doublehat v$
in $B(\eta^e_{r}, 2 r)\cap\Omega$ and since $x$ is in this set the result will follow by showing that $\doublehat v$ vanishes
at least as fast as $d(x, \partial\Omega)$ when $x \to \partial \Omega$.
However, this follows by the gradient bound in Lemma \ref{le:barrier_super} with constant depending on
$\lambda, \Lambda, n, \phi, r_e$ and $M = r^{-1} \sup_{B(w,6r)\cap \Omega} u$.
\end{proof}\\

%
\begin{figure}
\begin{center}
\includegraphics[width=14cm,height=9.1cm,viewport=0 0 460 365,clip]{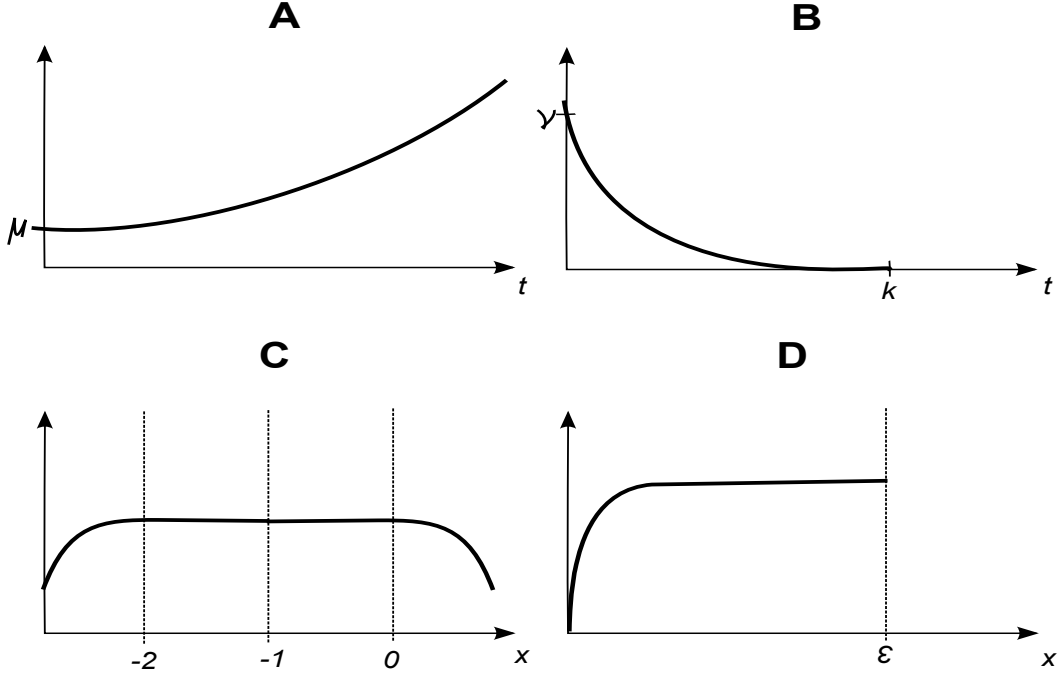}
\end{center}
\caption{(A) The function $g(t)$ in the proof of Lemma \ref{le:barrier_sub}, (B) $f(t)$ in the proof of Lemma \ref{le:barrier_super}, (C) the function $H(x)$ showing the necessity of assumption \eqref{eq:ass_phi_near} and (D) the function $F(x)$ showing the necessity of assumption \eqref{eq:ass_phi_inf_new}.}
\label{fig:functions}
\end{figure}

We now comment on the necessity of the assumptions in Theorem \ref{le:lower} and Theorem \ref{le:upper}.
Indeed, necessity of \eqref{eq:ass_phi_near} for the lower bound follows by the counterexample in Remark \ref{remark:counter1} as
the derivative of $1 - H$ approaches zero as the function itself vanishes near the origin.
To discuss the necessity of \eqref{eq:ass_phi_inf_new} for the upper bound,
assume that it does not hold.
Then there exist an $\varepsilon > 0$ and $C < \infty$ such that
\begin{align*}
 \int_{0}^{\varepsilon} f(t) dt \leq C  \quad\text{for all} \quad \nu > 0 
\end{align*}
where $f(t)$ solves the differential equation $f'(t) = - \phi(f(t))$ with initial condition $f(0) = \nu$.
Define
\begin{align*}
F(x) =   \int_{0}^{x} f(t) dt, \quad \text{whenever}\quad x \in  (0,\varepsilon)
\end{align*}
and note that then $F'(x) = f(x)$ and $F(x)$ satisfies the differential equation
$$
F'' + \phi(|F'|) = 0.
$$
Moreover, $F'(0) = f(0) = \nu$ which explodes as $\nu \to \infty$, see Figure \ref{fig:functions}.
This contradicts the upper bound in Theorem \ref{le:upper} and we conclude the following remark.

\begin{remark}\label{re:necessity-phi-inf}
Condition \eqref{eq:ass_phi_near} is necessary for the lower bound in Theorem \ref{le:lower} and
condition \eqref{eq:ass_phi_inf_new} is necessary for the upper bound in Theorem \ref{le:upper}.
\end{remark}

As assumption \eqref{eq:ass_phi_inf_new} is a bit technical we comment on some functions $\phi$ satisfying it.
To this end we restrict to $\phi(s) = s^k$, $k \geq 1$, and note that the solution to
$
f'(t) = - f(t)^k
$
with initial condition $f(0) = \nu$ yields
$$
f(t) = \left\{
\begin{array}{ll}
\nu e^{-t} & \text{if $k = 1$}, \\
\left((k - 1) t + \nu^{1-k}\right)^\frac{1}{1-k} & \text{if $k > 1$}.
\end{array}
\right.
$$
Therefore,
$$
\int_0^\epsilon f(t) dt
= \left\{
\begin{array}{ll}
\nu\left(1 - e^{-\epsilon}\right)  & \text{if $k = 1$},\\
\log(1 + \epsilon \nu) & \text{if $k = 2$}, \\
\frac{k-1}{2-k} \left( \nu^{2-k} - \left((k-1) \epsilon + \nu^{1-k}\right)^\frac{2-k}{1-k}\right) & \text{otherwise},
\end{array}
\right.
$$
and the above integral diverges for any $\epsilon > 0$ whenever $1 \leq k \leq 2$
and converges if $k > 2$.
Thus assumption \eqref{eq:ass_phi_inf_new} is satisfied if and only if $k \leq 2$
and it follows that condition \eqref{eq:ass_Cs^2} implies \eqref{eq:ass_phi_inf_new}.
Moreover, as
\begin{align*}
\int_{1}^{\infty} \frac{dt}{\phi(t)} = \infty
\end{align*}
ensures that \eqref{eq:ass_Cs^2} holds we have proved:

\begin{remark}\label{re:assumption-phi-inf}
The Keller-Osserman type condition \eqref{eq:ass_phi_inf} implies \eqref{eq:ass_Cs^2} which in turn implies \eqref{eq:ass_phi_inf_new}.
\end{remark}

Using the lower and upper estimates in Theorems \ref{le:lower} and \ref{le:upper} we conclude that
positive viscosity solutions vanishing on a portion of the boundary are comparable with the distance function near the boundary:

\begin{corollary}\label{corr:dist}
Let $\Omega \subset \mathbf{R}^n$ be a domain satisfying the ball condition with radius $r_b$.
Assume that \eqref{eq:ass_deg_ellipt}, \eqref{eq:ass_drift_super}, \eqref{eq:ass_drift_sub_ast}, \eqref{eq:ass_phi_near} and \eqref{eq:ass_phi_inf_new}  hold.
Let $w \in \partial \Omega$ and $0 < r < r_b$.
Assume that $u$ is a positive viscosity solution of \eqref{eq:main} in $\Omega \cap B(w, 6r)$,
satisfying $u = 0 = v$ on $\partial \Omega \cap B(w,6r)$.
Then there exists a constant $c$ such that
%
%
\begin{align*}
c^{-1}\, d(x, \partial \Omega) \leq u(x) \leq c \,d(x, \partial \Omega) \qquad \text{whenever} \quad x \in \Omega \cap B(w, r).
\end{align*}
The constant $c$ has dependence according to Theorems \ref{le:lower} and \ref{le:upper}.
\end{corollary}

An immediate consequence of Corollary \ref{corr:dist} is the Boundary Harnack inequality:

\begin{corollary}[Boundary Harnack inequality]\label{thm:harnack}
Let $\Omega$, $w$, $r$ and $u$ be as in Corollary \ref{corr:dist} and suppose that $v$ is another solution satisfying the same properties as $u$.
Then there exists a constant $c$ such that
\begin{align*}
\frac{1}{c} \leq \frac{u(x)}{v(x)} \leq c \qquad \text{whenever} \quad x \in \Omega \cap B(w, r).
\end{align*}
The constant $c$ has dependence according to Theorems \ref{le:lower} and \ref{le:upper}.
\end{corollary}

\section{A variable exponent eigenvalue problem}
\label{sec:variable-exp}

\setcounter{theorem}{0}
\setcounter{equation}{0}

In this section we discuss our results in the setting of the eigenvalue problem
\begin{align}\label{eq:eig}
\Delta_{p(\cdot)} := \nabla \cdot \left( |\nabla u|^{p(x)-2} \nabla u\right) = a|u|^{q(x)-2} u,
\end{align}
where $1 < p^- \leq p(x) \leq p^+< \infty$ is Lipschitz continuous, $q(x) \geq 2$ and $a \geq 0$.
Solutions of \eqref{eq:eig} are often considered in the following weak sense:
A function $u \in W^{1, p(x)}(\Omega)$ is a weak solution of \eqref{eq:eig} in $\Omega$ if
\begin{align*}
\int_{\Omega} \left( \langle |\nabla u|^{p(x)-2} \nabla u, \nabla \psi \rangle + |u|^{q(x)} u \psi \right) dx,
\end{align*}
whenever $\psi \in W_0^{1, p(x)}(\Omega)$, see e.g. Zhang \cite{Z05}.
Following arguments from Juutinen--Lukkari--Parviainen \cite{JLP10} and Julin \cite{J13} for variable exponent $p$-Laplace equation we conclude that continuous weak solutions of \eqref{eq:eig} are also viscosity solutions of
$$
F(x, u, Du, D^2u) = D^2 u + (p(x) - 2)\Delta_\infty u + \log |Du| \langle Dp Du\rangle + |u|^{q(x)} u,
$$
where $\Delta_\infty u$ denotes the infinity Laplace operator.
Moreover, following  Julin \cite{J13} we see that the above operator satisfies \eqref{eq:ass_ellipt}, \eqref{eq:ass_drift_super_old} and \eqref{eq:ass_drift_sub_ast} with
\begin{align}\label{hallojj}
\lambda = \min\{ 1, p^{-} - 1\}, \quad \Lambda = \max\{ 1, p^{+} - 1\}, \quad \phi(t) = C\left( |\log t| + 1\right) t  \quad \textrm{and} \quad
\gamma(t) = t,
\end{align}
where $\phi(t) = C\left( |\log t| + 1\right) t$ satisfies both the Osgood condition \eqref{eq:ass_phi_near} and the Keller-Osserman condition \eqref{eq:ass_phi_inf}. In conclusion:

\begin{corollary}
All Theorems in Sections \ref{sec:auxiliary}, \ref{sec:maximum} and \ref{sec:boundary} apply to continuous weak solutions of \eqref{eq:eig}.
The dependence of constants in the Theorems of Section \ref{sec:boundary} can be traced via \eqref{hallojj}.
\end{corollary}

We proceed by noting that existence of weak solutions for $\Delta_{p(\cdot)} = a|u|^{p(x)-2}u$ follows by Mendez \cite{Mendez}, with
\[
a = \frac{\int_\Omega |\nabla u(x)|^{p(x)}\,dx}{\int_\Omega |u(x)|^{p(x)}\,dx}.
\]
%
Put
\newcommand{\psip}{\psi_{p(\cdot)}}
\[
\psip(u) = |u|^{p(\cdot)-2}u.
\]
Concerning uniqueness we give the following proof of the comparison principle
based on Fleckinger-Pell\'e--Tak\'a\v{c} \cite[Proposition 4.1]{FPT}:

\begin{theorem}
Assume that $u, v \in W^{1,p(\cdot)}(\Omega)$, are weak solutions, respectively, of
\begin{align}\label{u-sol}
-\Delta_{p(\cdot)} u = a\psip(u) + f, \qquad u = g \textrm{ on } \partial\Omega,
\end{align}
and
\begin{align}\label{v-sol}
-\Delta_{p(\cdot)} v = a\psip(v) + f', \qquad v = g' \textrm{ on } \partial\Omega,
\end{align}
with $a<0$, $f \leq f'$ in $L^{p'(\cdot)}(\Omega)$ and $g \leq g'$ in 
$\Omega$. Then $u \leq v$ a.e. in $\Omega$.
\end{theorem}

\noindent
\begin{proof}
Set $w = u - v$ and $w^+ = \max\{w,0\}$, $w^- = \max\{-w,0\}$. Then $w = w^+ - w^-$ by definition, and $w^+ \in W^{1,p(\cdot)}(\Omega)$ with the trace $w^+ = 0$ on $\partial\Omega$.
By multiplying \eqref{u-sol} with $w^+$, integrating over $\Omega$, and applying integration by parts, we get
\begin{align*}
\int_\Omega |\nabla u|^{p(\cdot)-2} \langle \nabla u, \nabla w^+\rangle \,dx &= -\int_\Omega \nabla\cdot( |\nabla u|^{p(\cdot)-2}\nabla u  )w^+\,dx \\ &= a \int_\Omega |u|^{p(\cdot)-2}u w^+\,dx + \int_\Omega fw^+\,dx.
\end{align*}
Applying the same procedure to \eqref{v-sol}, subtracting the two equations, and remembering $f \leq f'$, we get
\begin{align*}
& \int_\Omega \left[|\nabla u|^{p(\cdot)}\nabla u - |\nabla v|^{p(\cdot)-2}\nabla v \right]\cdot \nabla w^+\,dx - a \int_\Omega \left[|u|^{p(\cdot)-2}u - |v|^{p(\cdot)-2}v\right]w^+\,dx \\ &= \int_\Omega (f-f')w^+\,dx \leq 0.
\end{align*}
Next consider the sets
\[
\Omega^+ = \{ x \in \Omega \mid u(x) > v(x) \} \quad \text{and} \quad \Omega^-_0 = \{ x \in \Omega \mid u(x) \leq v(x) \}.
\]
Note that $\Omega = \Omega^+ \cup \Omega^-_0$, and that in $\Omega^+$ we have $w^+ = u-v > 0$ and $\nabla w^+ = \nabla u - \nabla v$, while in $\Omega^-_0$ we have $w^+ = \nabla w^+ = 0$ (Gilbarg-Trudinger \cite[Lemma 7.6]{GT}).
Using this the previous equation can be written as
\begin{align}\label{contra}
\nonumber & \int_{\Omega^+} \left(|\nabla u|^{p(\cdot)-2}\nabla u - |\nabla v|^{p(\cdot)-2}\nabla v\right]\cdot \nabla (u-v)\,dx \\& - a \int_{\Omega^+} \left[|u|^{p(\cdot)}u - |v|^{p(\cdot)-2}v\right](u-v)\,dx \\
\nonumber &= \int_{\Omega^+} (f-f')w^+\,dx \leq 0.
\end{align}
Next assume that our claim of $u \leq v$ does not hold in a set of positive measure, that is, assume that $|\Omega^+| > 0$. We will show this leads to a contradiction by showing that in this case the left-hand side of \eqref{contra} will be positive.
Firstly, since $\psip$ is strictly monotonously increasing for a.e. $x \in \Omega$, and $a<0$, we have
\[
- a \int_{\Omega^+} \left[|u|^{p(\cdot)-2}u - |v|^{p(\cdot)-2}v\right](u-v)\,dx > 0.
\]
Secondly, a similar argument by the monotonicity of $A(\xi) = |\xi|^{p(x)-2}\xi$ gives
\[
\int_{\Omega^+} \left[\psip(\nabla u) - \psip(\nabla v)\right]\cdot \nabla (u-v)\,dx \geq 0.
\]
This gives the contradiction, and the result.
\end{proof}\\


We end the paper by noting that in case of the eigenvalue problem \eqref{eq:eig},
our auxiliary functions constructed in Section \ref{sec:auxiliary} can be replaced by those constructed in
Adamowicz--Lundstr\"om \cite[Lemma 4.1]{AL16}.
Indeed, with restriction to \eqref{eq:eig} we can prove all our main results using the following Lemma:


\begin{lemma}
Let $p(x) \in (p^-, p^+)$ be Lipschitz continuous, $q(x) \in (q^-, q^+)$, $M > 0$ and define
\[
\hat{u}(x) = \frac{M}{e^{-\mu} - e^{-4\mu}} \left(  e^{-\mu} - e^{-\mu \frac{|s-y|^2}{r^2}}\right) \quad \text{whenever} \quad x \in B(y,2r)\setminus B(y,r).
\]
Then there are constants $r_* = r_*(p^-,\|p\|_{L^\infty})$ and $\mu_* = \mu_*(p^+,p^-,q^-, q^+,n,\|p\|_{L^\infty},M)$
such that for $r \leq r_*$ and $\mu \geq \mu_*$ it holds that $\hat{u}(x) = M$ on $\partial B(y,2r)$,
$\hat{u}(x) = 0$ on $\partial B(y,r)$ and $\hat{u}$ is a classical supersolution to \eqref{eq:eig}.
\end{lemma}

\noindent
\begin{proof}
The proof in Adamowicz--Lundstr\"om \cite[Lemma 4.1]{AL16} hinges on calculating $\Delta_{p(\cdot)} u$ for $u = A  e^{-\mu \frac{|s-y|^2}{r^2}} + B$ and picking the right $A$ and $B$. The estimates (4.1), (4.6), (4.8) in \cite{AL16} carry forward the need for $\Delta_{p(\cdot)} u \leq (\geq) 0$, and end in requiring the bounding of $\mu$ and $r$ so that
\begin{align}\label{crux}
\nonumber &\mu ( 8r\|\nabla p\|_{L^\infty} - 2(p^- - 1)) \\
&+ 2r \|\nabla p\|_{L^\infty} \left( \log \left( \frac{4}{1-e^{-3\mu}}\right)  + |\log M| + |\log r| \right) \\
\nonumber &+ n + p^+ - 2 \leq 0.
\end{align}
The first term is negative if $r \leq \frac{p^- - 1}{4\|\nabla p\|_{L^\infty}}$, and with an upper bound on $r$, $r |\log r|$ is increasing, meaning the second term does not blow up for smaller $r$. The inequality then holds for a big enough positive $\mu$.

In our case we have $\Delta_{p(\cdot)} u \leq a |u|^{q(x) - 2} u$  (in place of $\Delta_{p(\cdot)} u \leq 0$),
which enters on the RHS of \eqref{crux}.
We stick to the choices of $A$ and $B$ in \cite{AL16},
which give $0 \leq u \leq M$.

When $a \geq 0$, the inequality holds as in Adamowicz--Lundstr\"om.
We remark that this also holds when $a < 0$, since we can move the extra term to the LHS of \eqref{crux} and note that
\[
-a |u|^{q(x)-2}u
\leq -a \left| \frac{M}{e^{-\mu} - e^{-4\mu}} (e^{-\mu} - e^{-\frac{|x-y|^2}{r^2}}) \right|^{q(x)-1}
\leq -a \max\{ M^{q^+-1}, M^{q^--1}\}
\]
since
\[
0 < \frac{e^{-\mu} - e^{-\frac{|x-y|^2}{r^2}}}{e^{-\mu} - e^{-4\mu}} < 1.
\]
This means the additional term is bounded so that with $\mu$ large enough, the LHS of \eqref{crux} will still be non-positive as in \cite{AL16}. This completes the proof of the Lemma.
\end{proof}\\

\noindent
{\bf Acknowledgement.} We thank Tomasz Adamowicz for useful discussions, ideas and comments.
 The work of Niklas L. P. Lundstr\"om was partially supported by the Swedish research council grant 2018-03743.


\end{document}